\documentclass{birkjour}
 \newtheorem{thm}{Theorem}[section]
 \newtheorem{cor}[thm]{Corollary}
 \newtheorem{lem}[thm]{Lemma}
 \newtheorem{prop}[thm]{Proposition}
 \theoremstyle{definition}
 \newtheorem{defn}[thm]{Definition}
 \theoremstyle{remark}
 
 \newtheorem*{ex}{Example}
 \numberwithin{equation}{section}

\begin{document}

%
%
%
%
%
%
%
%
%

\title[ Biharmonic  Hypersurfaces ]
 {Survey On The Biharmonic Hypersurfaces In Terms Of the Induced Metric Of Tensor Ricci}

\author[N. Mosadegh]{N. Mosadegh}

\address{Department of Mathematics,
 Azarbaijan Shahid Madani University,\\
Tabriz 53751 71379, Iran}

\email{n.mosadegh@azaruniv.ac.ir}

\author{E. Abedi}
\address{Department of Mathematics,
 Azarbaijan Shahid Madani University,\\
Tabriz 53751 71379, Iran}
\email{esabedi@azaruniv.ac.ir}


\date{January 1, 2004}
\dedicatory{}

\begin{abstract}
In this article, we study the biharmonic hypersurfaces in the Sasakian space form with the induced metric of tensor Ricci. We find the existence necessary and sufficient condition of the biharmonic hypersurfaces there. We show that the biharmonic Hopf hypersurfaces are minimal where gradient of the mean curvature is in direction of the structural vector fields. Furthermore, we prove that does not exist any biharmonic Hopf hypersurfaces when the gradient of the mean curvature is a principal direction.
\end{abstract}
\footnotetext[1]{The first author is as corresponding author.}
\maketitle
\section{Introduction}
\label{intro}
Consider an immersion $\psi: (M, g)\rightarrow(N, h)$ between two Riemannian manifolds, which is known as the critical point of the energy functional
\begin{eqnarray*}
E:C^{\infty}(M, N)\rightarrow R,\nonumber\\
 E(\psi)=\int_{M}|\tau(\psi)|^2d\vartheta,
\end{eqnarray*}
called a biharmonic map and quivalently $M$ is said the biharmonic submanifold, where $\tau(\psi)=\textsf{trace} \nabla d\psi$ known the tension field of $\psi$ see (\cite{Ee,Ee1}). Vanishing of the tension field characterizes the harmonic map as well. The Euler-Lagrangian equation associated to the energy functional, written as
 \begin{eqnarray}\label{7.0}
 0=\tau_2(\psi)&=&-\Delta \tau(\psi)- \textsf{trace}R^N(d\psi(.), \tau(\psi))d\psi(.),
 \end{eqnarray}
 see \cite{Jiang}. Here $\Delta=- \textsf{trace} \nabla^2$ stands for the Laplace-Beltrami operator, where $\nabla$ is a induced connection in the pull back bundle $\varphi^{-1}(TN)$ and $R^{N}$ is the curvature operator on $N$ in which defined by $R^N(X,Y)=[\nabla_X, \nabla_Y]- \nabla_{[X, Y]}$ for all $X$ and $Y$ tangent to $N$. We recall, it is known that vanishing the tension field is characterization of harmonic map, therefore, the curiosity is about nonharmonic biharmonic maps which are called proper biharmonic.

 In 1980s, the characterization formula of the biharmonic Riemannian immersion in the Euclidean space, that is, $\Delta H=0$, where $H$ denotes the mean curvature vector field, was introduced by Chen \cite{Ch2}, independently. Then, it was shown that every biharmonic hypersurface in the Euclidean space $E^n$ is a minimal hypersurface where $n=3,4,5$ (see\cite{Ch1}, \cite{Has},\cite{Gat}). Yu Fu \cite{YU1} proved that every biharmonic hypersurface with three distinct principal curvatures in the Euclidean space is minimal.
Also, with respect to the number of distinct principal curvature concerning the biharmonic hypersurfaces, the authors reached to the following results. Every biharmonic Hopf hypersurface in the complex Euclidean space is a minimal one. There exist no proper biharmonic Ricci saliton hypersurface in the Euclidean space in such away that the potential vector field is a principal vector in the direction of the gradient of the mean curvature vector field. Every biharmonic hypersurface equipped with the recurrent operator in the Euclidean space is minimal as well. The biharmonic Hopf QR-hypersurfaces are minimal ones in the quaternionic Euclidean space (see\cite{Abedi3,Abedi4,Abedi5,Abedi6}).

Additionally, in space of the nonconstant sectional curvature, several classification concerning the proper biharmonic has been investigated. For example, classified all the proper biharmonic Hopf cylinder in $3$-dimensional Sasakian space forms and all the proper biharmonic homogeneous real hypersurfaces in the complex projective space \cite{JI,JI1}. Also, the authors have already obtained the biharmonic conditions of CMC hypersurfaces in the special warped product, and the nonexistence result of the biharmonic Hopf hypersurfaces in the Sasakian space form associated to the Tanaka-Webster connection, see \cite{Abedi2,AAAA}.

TThe aim of this paper is studying about the biharmonic hypersurfaces in the Sasakian space form, where the ambient manifold is equipped with the induced metric of the tensor Ricci. We reach the necessary and sufficient condition about the existence of the biharmonic hypersurfaces there and call them Ricci-biharmonic hypersurfaces. In the first part of this paper, we present an example of a biharmonic surface in the Sasakian space form $R^3(-3)$ which is minimal. we obtain the nonexistent result about the Ricci-biharmonic Hopf hypersurfaces, where the \textsf{grad}H is a principal direction. Also, we show a Ricci-biharmonic Hopf hypersurface is minimal where the $\mathsf{grad}$H is in the direction of structural vector fields.

\section{Preliminary}
Throughout this paper all the objects are considered to be smooth. In this section we start with the notations of the Sasakian structure that will be used throughout the paper. Let $\Gamma(TM)$ and $\Gamma(TM^{(p,q)})$ denote the set of all the vector fields and tensor fields of type $(p,q)$ on manifold $M$, respectively. Let $\nabla$ stands for the Levi-Civita connection on $M$ and $g\in \Gamma(TM^{(0,2)})$. Take the triple $(\varphi, \xi, \eta)$, which $\varphi \in \Gamma(TM^{(1,1)}), \eta\in \Gamma(TM^{(1,0)})$ and $\xi\in \Gamma (TM)$, is known as an almost contact structure on the odd dimensional manifold $M^{2m+1}$, provided that the following equations hold for any $X, Y\in \Gamma(TM^{2m+1})$
\begin{eqnarray}
\varphi^2(X)=-X+\eta(X),\ \ \
\eta(\xi)=1,\ \ \ \varphi(\xi)=0.
\end{eqnarray}
Now, a Riemannian metric $g$ is endowed on $M^{2m+1}$, then $(\varphi,\xi,\eta,g)$ is called almost contact metric structure if
\begin{eqnarray}
\eta(X)=g(\xi,X), \ \
g(\varphi X, \varphi Y)&=&g(X, Y)-\eta(X)\eta(Y).
\end{eqnarray}
An almost contact metric structure on $M^{2m+1}$ is a contact metric structure where $d\eta(X,Y)=g(X, \varphi Y)$. Now, a contact metric manifold is named a Sasakian manifold if and if
\begin{eqnarray}\label{7.9}
(\nabla_X \varphi)Y=g(X,Y)\xi- g(Y,\xi)X,
\end{eqnarray}
where $X,Y\in \Gamma(TM^{2m+1})$ and $\nabla$ stands for the Levi-Civita connection on $M^{2m+1}$. For a Sasakian manifold it is known that
\begin{eqnarray}
\nabla_X \xi= -\varphi X.
\end{eqnarray}
At point $x$ of a Sasakian manifold $M^{2m+1}$ a $\varphi$-section means a 2-plane spanned by $\{X, \varphi X\}$ in $T_x(M^{2m+1})$, where $X$ is an unit tangent vector orthogonal to $\xi$. A Sasakian manifold $(M^{2m+1}, \varphi, \xi, \eta, g)$ has constant $\varphi$-sectional curvature $c\in R$, provided that its curvature tensor satisfies
\begin{eqnarray}\label{7.15}
R(X, Y)Z &=& -\frac{c-1}{4}\{\eta(Z)[\eta(Y)X - \eta(X)Y] \nonumber \\
& & + [g(Y, Z)\eta(X)- g(X, Z)\eta(Y)]\xi\nonumber \\
& & + g( \varphi X, Z)\varphi Y + 2g(\varphi X, Y)\varphi Z - g(\varphi Y,Z)\varphi X\}\\
& & + \frac{c+3}{4}\{ g(Y, Z)X - g(X, Z)Y\}, \nonumber
\end{eqnarray}
 where $X, Y$ and $Z\in \Gamma(TM^{2m+1})$. A Sasakian manifold of constant $\varphi$-sectional curvature $c$ is called a Sasakian space form and determined by $\overline{M}^{2m+1}(c)$.

At the end of this section, we construct an example of a Riemannian biharmonic surface in the Sasakian space form $R^3(-3)$, which is harmonic($H=0$) as well.
 \begin{ex}
Let $R^{3}$ be a hypersurface in the Euclidean space $R^4$. Let $J$ be the standard almost complex structure in $R^4$ and set $\xi=-J N$, where $N$ is an unit normal vector field of $R^3$. Define $\varphi$ by $\pi o J$, where $\pi$ is the natural projection of the tangent space of $R^4$ in to the tangent space of $R^3$. Let $(x, y, z)$ be the Euclidean coordinate in $R^3$, we consider
\begin{eqnarray}
\eta= \frac{1}{2}(dz-ydx),\ \
g=\eta \otimes \eta +\frac{1}{4}(dx^2+ dy^2),\nonumber\\ \ \varphi(X\frac{\partial}{\partial x}+ Y\frac{\partial}{\partial y}+Z\frac{\partial}{\partial z})=Y\frac{\partial}{\partial x}-X\frac{\partial}{\partial y}+Yy\frac{\partial}{\partial z}
\end{eqnarray}
where $\xi=2\frac{\partial}{\partial z}$. Then $(R^3, \varphi, \eta, \xi, g)$ is called a Sasakian space form where its $\varphi$-sectional curvature is $c=-3$. Let $f\in C^{\infty}(R^3(-3))$ defines $f(x,y,z)=z$, then we consider the level set of $f$ like $M^2=f^{-1}(0)= \{(x,y,z)\in R^3; z=0\}$ which is claimed as a minimal surface (as well as biharmonic) of $R^{3}(-3)$. In order to show this property, we choose an appropriate orthonormal frame field on $R^3(-3)$ such as
\begin{eqnarray}
e_1= 2(\frac{\partial}{\partial x}+ y \frac{\partial}{\partial z}),\ \ e_2=-2\frac{\partial}{\partial y}\ \ ,e_3= 2\frac{\partial}{\partial z}
\end{eqnarray}
then we calculate $\textsf{grad} f=\sum_{i=1}^3 e_i(f)e_i = 2(ye_1+e_3)$. So, $N=\frac{\textsf{grad} f}{|\textsf{grad }f|}=\frac{1}{\sqrt{1+y^2}}(ye_1+e_3)$ is an unit normal vector on $M^2$. Also, $-\varphi N=V=-\frac{y}{\sqrt{1+y^2}}e_2 $ is in $\Gamma(TM^2)$. With respect to these last vectors we take an orthonormal frame field $\{E_1=\frac{V}{|V|}=-e_2,\ \ E_2= -\frac{1}{\sqrt{1+y^2}}(e_1-ye_3)\}$ on $M^2$. Some easy computations show the following bracket relations, which we need to calculate the Weingarten operator $A$ of $M^2$ in the Sasakian Space $R^3(-3)$, as following
\begin{eqnarray*}
&&[e_1,e_2]=2e_3,\ \ [e_1, e_3]=0,\ \ [e_2, e_3]=0\\
&&\nabla_{e_1}e_2=-\nabla_{e_2}e_1 =e_3,\ \ \nabla_{e_1}e_3=\nabla_{e_3}e_1=-e_2,\ \ \nabla_{e_2}e_3=\nabla_{e_3} e_2= e_1
\end{eqnarray*}
where $\nabla$ is the Levi-Civita connection on $R^3(-3)$. After all, we can calculate
\begin{eqnarray*}
-AE_1= \nabla_{E_1} N= \frac{1-y^2}{1+y^2}E_2,\\
-AE_2=\nabla_{E_2} N=\frac{1-y^2}{1+y^2}E_1
\end{eqnarray*}
then we have
\begin{eqnarray*}
A=\left(
  \begin{array}{cc}
    0 & \frac{y^2-1}{1+y^2} \\
    \frac{y^2-1}{1+y^2} & 0 \\
  \end{array}
\right).
\end{eqnarray*}
So, the shape operator matrix of $M^2$ presents that the mean curvature $|H|=0$. In other words, $xy$-plane is a minimal or harmonic surface in the Sasakian space form $R^3(-3)$.
\end{ex}

\section{Ricci Biharmonic Hypersurfaces}

In this paper the tensor Ric(X,Y) is considered a metric in a Sasakian space form $(\overline{M}^{(2m+1)}(c),\overline{g})$ and written as
 \begin{eqnarray}\label{7.1}
 \mathsf{Ric}(X,Y)&=&\frac{(m+1)c+3m-1}{2}\overline{g}(X, Y)\nonumber\\
 && + \frac{(m+1)(1-c)}{2}\eta(X)\eta(Y),
 \end{eqnarray}
 for $X$ and $Y \in \Gamma(T\overline{M}^{(2m+1)}(c))$.

By the Koszul formula the connection $\nabla^{\ast}$ associated to the induced metric of the tensor Ricci is characterized here 
\begin{lem}\label{7.14}
Let $\overline{M}^{2m+1}(c)$ be a Sasakian space form with induced metric of the tensor Ricci. Then, the associated connection $\nabla^{*}$ satisfies
\begin{eqnarray}
\nabla^{*}_X Y= \overline{\nabla}_X Y-\frac{k_1}{k_2}(\eta(X)\varphi Y+\eta(Y)\varphi X),
\end{eqnarray}
where $\overline{\nabla}$ denotes the Levi-Civita connection on $\overline{M}^{2m+1}(c)$, $k_1=\frac{(m+1)(1-c)}{2}$ and $k_2=\frac{(m+1)c+3m-1}{2}$.
\end{lem}
\begin{proof}
Suppose that $\overline{g}$ and $\overline{\nabla}$ are the Riemannian metric and the Levi-Civita connection on $\overline{M}^{2m+1}(c)$, respectively. From the Koszul formula and the equation $(\ref{7.1})$ we have
\begin{eqnarray}\label{7.2}
\mathsf{Ric}(\nabla^{*}_Y Z, X)&=& Y \mathsf{Ric}(Z, X)+ Z\mathsf{Ric}(X, Y)+ X\mathsf{Ric}(Y, Z) \nonumber\\
&& -\mathsf{Ric}(Y, [Z, X])+ \mathsf{Ric}(Z, [X, Y])+ \mathsf{Ric}(X, [Y,Z]) \nonumber\\
&=& k_1(\eta(\overline{\nabla}_Y Z)\eta(X)- \overline{g}(\varphi Y, X)\eta(Z)- \overline{g}(\varphi Z, X)\eta(Y)) \nonumber\\ &&+k_2\overline{g}(\overline{\nabla}_Y Z, X),
\end{eqnarray}
then the equations $(\ref{7.1})$ and $(\ref{7.2})$ yield
\begin{eqnarray*}\label{7.5}
k_1\overline{g}(\nabla^{*}_Y Z, \xi)\xi+k_2\nabla^{*}_Y Z&=& k_1\big(\overline{g}(\overline{\nabla}_Y Z, \xi)\xi-\eta(Z)\varphi Y-\eta(Y)\varphi(Z)\big)\nonumber\\
&&+k_2\overline{\nabla}_Y Z,
\end{eqnarray*}
which shows
\begin{eqnarray*}\label{7.4}
\overline{g}(\nabla^{*}_Y Z, \xi)=\overline{g}(\overline{\nabla}_Y Z, \xi),
\end{eqnarray*}
then
\begin{eqnarray*}
k_2\nabla^{*}_Y Z= -k_1(\eta(Z)\varphi Y+\eta(Y)\varphi Z)+k_2\overline{\nabla}_Y Z,
\end{eqnarray*}
so, we reach the result.
\end{proof}

The bitension field of an isometric immersion $\psi: M^{2m}\rightarrow \overline{M}^{2m+1}$ with the connection $\nabla^{*}$ holds
\begin{eqnarray}
\tau_2^{\star}(\psi)= -\Delta^{\star}(H)- \textsf{trace} R^{\star}(d\psi(.), H)d\psi(.),
\end{eqnarray}
here $\Delta^{\star}$ stands for the Laplace-Beltrami operator on sections of the pull back bundle $\psi^{-1}(\Gamma (T\overline{M}^{2m+1}(c)))$ and $R^{\star}$ denotes the curvature operator on the Sasakian space form $\overline{M}^{2m+1}(c)$ where the following sign conventions hold
\begin{eqnarray}
\Delta^{\star} X = - \textsf{trace} \nabla ^{\star^2} X, \ \ \ \ \forall X \in \psi^{-1}(\Gamma (T\overline{M}^{2m+1})),\label{7.19}\\
 R^{\star}(X, Y)=[\nabla^{\star}_X, \nabla^{\star}_Y]- \nabla^{\star}_{[X, Y]}\ \ \ \ X, Y\in{\Gamma (T\overline{M}^{2m+1})}.\label{7.13}
\end{eqnarray}

\begin{defn}\label{7.6}
A hypersurface $M^{2m}$ in the Sasakian space form $\overline{M}^{2m+1}(c)$ with the induced metric of tensor Ricci is called Ricci-biharmonic hypersurface if $\tau_2^{\star}(\psi)=0$.
\end{defn}
Let $(\overline{M}^{2m+1}(c), \varphi, \xi, \eta, \overline{g})$ be a Sasakian space form. We suppose that $\xi$ and $V=-\varphi N$ are the tangent vector fields on $M^{2m}$, where $N$ is a local unit normal vector on $M^{2m}$ as well.
\begin{lem}\label{7.20}
Let $M^{2m}$ be an immersed hypersurface in the Sasakian space form $\overline{M}^{2m+1}(c)$ with the induced metric of tensor Ricci, isometrically. Then
\begin{eqnarray}
\Delta^{*}H= \Delta H +\big(3\frac{k_1}{k_2} + (\frac{k_1}{k_2})^2\big)H- 2 \frac{k_1}{k_2}\eta(\mathsf{grad}|H|)V,
\end{eqnarray}
where $H$ is the mean curvature vector field of $M^{2m}$ in the Sasakian space form.
\end{lem}
\begin{proof}
Let $\nabla^{\star}$, $\overline{\nabla}$ and  $\nabla$  denote the connections associated to induced metric of the tensor Ricci, Riemannian metric $\overline{g}$ and the Levi-Civita connection on $\overline{M}^{2m+1}(c)$ and  $M^{2m}$, respectively. We consider a parallel local orthonormal frame field $\{e_\alpha\}_{\alpha=1}^{2m}$ at $p \in M^{2m}$, the equation $(\ref{7.9})$ and the Weingarten operator $\overline{\nabla}_{e_\alpha}H= \nabla^{\perp}_{e_\alpha}H-A_H e_\alpha,$ then
 \begin{eqnarray}\label{7.12}
\Delta^{\star} H &=& -\sum_{\alpha=1}^{2m} \nabla^{\star}_{e_\alpha}\nabla^{\star}_{e_\alpha} H \nonumber\\
                 &=& -\sum_{\alpha=1}^{2m} \nabla^{\star}_{e_\alpha}\big(\overline{\nabla}_{e_\alpha} H - \frac{k_1}{k_2}( \varphi (H)\eta(e_\alpha) + \varphi (e_\alpha) \eta(H))\big)\nonumber\\
                 &=& -\sum_{\alpha=1}^{2m} \big \{\overline{\nabla}_{e_\alpha}\overline{\nabla}_{e_\alpha} H - \frac{k_1}{k_2}( \varphi (\overline{\nabla}_{e_\alpha} H)\eta(e_\alpha) + \varphi (e_\alpha) \eta(\overline{\nabla}_{e_\alpha} H))\nonumber\\
                 &&-\frac{k_1}{k_2}\eta(e_\alpha)(\overline{\nabla}_{e_\alpha} \varphi H- \frac{k_1}{k_2}\varphi^2 H \eta(e_\alpha))\big\} \nonumber\\
                 &=& -\sum_{\alpha=1}^{2m} \{\overline{\nabla}_{e_\alpha}\overline{\nabla}_{e_\alpha} H +\frac{k_1}{k_2}\eta(A_H e_\alpha)\varphi e_\alpha\}\nonumber\\
                  && - 2 \frac{k_1}{k_2}(\eta(\mathsf{grad}|H|)V+ \varphi A_H \xi)+(\frac{k_1}{k_2})^2 H.
\end{eqnarray}
We consider the following equations 
\begin{eqnarray*}\label{7.10}
\sum_{\alpha=1}^{2m} \eta(\overline{\nabla}_{e_\alpha} H)\varphi e_\alpha = \sum_{\alpha=1}^{2m} g(\overline{\nabla}_{e_\alpha}H, \xi)\varphi e_\alpha
= \sum_{\alpha=1}^{2m}g(H, \varphi e_\alpha)\varphi e_\alpha,
\end{eqnarray*}
on the one hand we have
\begin{eqnarray*}\label{7.11}
\sum_{\alpha=1}^{2m} \eta(\overline{\nabla}_{e_\alpha} H)\varphi e_\alpha = \sum_{\alpha=1}^{2m} g(\overline{\nabla}_{e_\alpha}H, \xi)\varphi e_\alpha
=-\sum_{\alpha=1}^{2m}g(A_He_\alpha , \xi)\varphi e_\alpha,
\end{eqnarray*}
consequently
\begin{eqnarray}\label{7.18}
A_H \xi=-|H|V.
\end{eqnarray}
where $-\varphi N= V$. By the above we calculate the first and second terms of $(\ref{7.12})$ and one obtain 
\begin{eqnarray}\label{7.17}
\sum_{\alpha=1}^{2m}g(A_H e_\alpha, \xi)\varphi e_\alpha &=& \sum_{\alpha, \beta=1}^{2m} g(A_H e_\alpha, \xi)(g(\varphi e_\alpha, e_\beta)e_\beta + g(\varphi e_\alpha, N)N)\nonumber\\
&=& -|H|\sum_{\alpha, \beta=1}^{2m} g(e_\alpha, V)(-g(e_\alpha,\varphi e_\beta)e_\beta - g(e_\alpha, \varphi N)N)\nonumber\\
&=&-|H|\sum_{\beta=1}^{2m} (g(V,\varphi e_\beta)e_\beta + g(V,V)N)\nonumber\\
&=&-H,
\end{eqnarray}
also  
\begin{eqnarray}
\overline{\nabla}_{e_\alpha} \overline{\nabla}_{e_\alpha} H&=&\nabla^{\perp}_{e_\alpha}\nabla^{\perp}_{e_\alpha} H\nonumber \\ &&-A_{\nabla^\perp_{e_\alpha}H}e_\alpha-\nabla_{e_\alpha}A_H(e_\alpha)-B(e_\alpha, A_H(e_\alpha)),
\end{eqnarray}
hence, we obtain
\begin{eqnarray}\label{7.16}
\Delta H=-\sum_{\alpha=1}^{2m}\overline{\nabla}_{e_\alpha}\overline{\nabla}_{e_\alpha} H &=&
-\Delta^{\perp} H + \textsf{trace}B(.,A_H .)\nonumber\\
&&+ \textsf{trace} A_{\nabla_{(.)}^{\perp}H}(.)+ \textsf{trace}\nabla_{(.)} A_H(.).
\end{eqnarray}
After all, take the equations $(\ref{7.18})$, $(\ref{7.17})$ and $(\ref{7.16})$  then replace them in $(\ref{7.12})$ we get the result.
\end{proof}
In the proof of the main results we need the following lemma.
\begin{lem}\label{7.21}
Let $\psi: M^{2m} \longrightarrow \overline{M}^{2m+1}(c)$ be an isometric immersion of $2m$-dimensional hypersurface $M^{2m}$ in the Sasakian space form $\overline{M}^{2m+1}(c)$ with induced metric of the tensor Ricci. Then
\begin{eqnarray}
\mathsf{trace}R^{*}(d\psi(.),H)d\psi(.)= kH,
\end{eqnarray}
where $k=-\frac{-9 - 11 m + 18 m^2 + c^2 (3 + 5 m + 2 m^2) +
  2 c (3 + 11 m + 6 m^2)}{2 (-1 + c + 3 m + c m)^2}$ and $H$ denotes the mean curvature vector field of $M^{2m}$ in $\overline{M}^{2m+1}(c)$.
\end{lem}
\begin{proof}
We calculate the curvature tensor $R^{\star}$ associted to the connection $\nabla^{\star}$. So, for $X , Y$ and $Z \in \Gamma(T\overline{M}^{2m+1}(c))$ we have
\begin{eqnarray}\label{Ab n}
R^{*}(X, Y)Z&=& \nabla^{*}_X \nabla^{*}_Y Z -\nabla^{*}_Y \nabla^{*}_X Z-\nabla^{*}_{[X,Y]} Z\nonumber\\
&=&\overline{R}(X, Y)Z - \frac{k_1}{k_2}\overline{g}(Z, \overline{\nabla}_X \xi)\varphi Y + \frac{k_1}{k_2}\overline{g}(Z, \overline{\nabla}_Y \xi)\varphi X\nonumber\\
&&+\frac{k_1}{k_2}\big(\overline{g}(X, \overline{\nabla}_Y \xi)-\overline{g}(Y, \overline{\nabla}_X \xi)\big)\varphi Z  \nonumber\\ &&+\frac{k_1}{k_2}\big(\eta(X)\overline{g}(Y,Z)-\eta(Y)\overline{g}(X,Z)\big)\xi\nonumber\\
&&+ (2\frac{k_1}{k_2}+ (\frac{k_1}{k_2})^2)(\eta(Y)\eta(Z)X -\eta(X)\eta(Z)Y),
\end{eqnarray}
where $\overline{R}$ stands for the curvature tensor of a Sasakian space form associated to the Riemannian metric $\overline{g}$. Now, we choose an appropriate orthonormal frame field $\{\xi, V\} \cup\{e_\alpha\}_{\alpha=1}^{2m-1}$ on $\overline{M}^{2m+1}(c)$, in such away that $\{ \xi^{*}, V^{*}\}\cup\{e_\alpha^{*}\}_{\alpha =1}^{2m-1}$, where $e^{*}_\alpha=\frac{1}{\surd k_2}e_\alpha,\ \ \xi^{*}=\frac{1}{\surd (k_1+ k_2)}\xi$ and $ V^{*}= \frac{1}{\surd k_2}V$, on $\overline{M}^{2m+1}(c)$ is an orthonormal frame field as well. Then by the above one obtain
\begin{eqnarray}
\textsf{trace}R^{*}((.), H)(.)&=& \sum_{\alpha=1}^{2m-1} R^{*}(e^{*}_\alpha, H)e^{*}_\alpha + R^{*}(\xi^{*}, H)\xi^{*} + R^{*}(V^{*}, H)V^{*},\nonumber
\end{eqnarray}
such that
\begin{eqnarray}
\sum_{\alpha=1}^{2m-1} R^{*}(e^{*}_\alpha, H)e^{*}_\alpha = \frac{1}{k_2}\sum_{\alpha=1}^{2m-1}\overline{R}(e_\alpha, H)e_\alpha=-\frac{c+3}{4k_2}(2m-1)H,\nonumber
\end{eqnarray}
and
\begin{eqnarray*}
R^{*}(\xi^{*}, H)\xi^{*}=\frac{1}{k_1 + k_2}(\overline{ R}(\xi, H)\xi -((\frac{k_1}{k_2})^2+ 2\frac{k_1}{k_2})H),
\end{eqnarray*}
and
\begin{eqnarray*}
R^{*}(V^{*}, H)V^{*}=\frac{1}{k_2}\big(\overline{R}(V, H)V + 3\frac{k_1}{k_2}\big),
\end{eqnarray*}
where $\overline{ R}(\xi, H)\xi = -H$ and $\overline{ R}(V, H)V = -cH$.
Then put all the above terms together, we obtain the result.
\end{proof}

Now, we can have the principal result.
\begin{thm}\label{7.7}
Let $\psi: M^{2m} \longrightarrow \overline{M}^{2m+1}(c)$ be an isometric immersion of $2m-$ dimensional hypersurface $M^{2m}$ in the Sasakian space form $\overline{M}^{2m+1}(c)$ with the induced metric of the Ricci tensor. Then $M^{2m}$ is a Ricci-biharmonic hypersurface if and only if
\begin{eqnarray}
\left\{
  \begin{array}{ll}
    \hbox{$\Delta^{\perp}H = (|B|^2 +l)H $} \\
     \hbox{A $(\textsf{grad}|H|) -\frac{k_1}{k_2} \eta(\textsf{grad} |H|)V + m |H| \textsf{grad}|H| =0$,\nonumber}
  \end{array}
\right.
\end{eqnarray}
where $l=\frac{-5-27m-2m^2+c^2(7+13m+6m^2)+c(-2+30m+28m^2)}{2(-1+c+3m+cm)^2}$ is constant, $ A$ and $B$ denote  the shape operator and second fundamental form of $M^{2m}$ in $\overline{M}^{2m+1}(c)$, respectively.
\end{thm}
\begin{proof}
Take the definition \ref{7.6} the hypersurface $M^{2m}$ is a Ricci-biharmonic hypersurface if 
\begin{eqnarray*}
0 = \tau_2^{\star}(\psi)
=- \Delta^{\star} H- \textsf{trace} R^{\star}(d\psi(.), H)d\psi(.),
\end{eqnarray*}
then the Lemma $\ref{7.20}$ and Lemma $\ref{7.21}$ follow
\begin{eqnarray*}\label{7.22}
0 &=& \tau_2^{\star}(\psi)\nonumber\\
&=&-\Delta H -(3\frac{k_1}{k_2}+ (\frac{k_1}{k_2})^2)H+ 2 \frac{k_1}{k_2}\eta(grad|H|)V-kH.
\end{eqnarray*}
 Thus, by the above separate the tangent and normal parts we get the result. 
\end{proof}

For the Ricci-biharmonic hypersurfaces with the constant mean curvature the Theorem \ref{7.7} follows
\begin{cor}
Let $M^{2m}$ be a Ricci-biharmonic hypersurface with the constant mean curvature. Then the $\varphi$-sectional curvature satisfies $c\in (-5, \frac{1}{2})$.
\end{cor}
\begin{proof}
Consider the case of Ricci-biharmonic hypersurfaces where the mean curvature vector field is constant $\neq 0$, then the Theorem $\ref{7.7}$ leads $\vert B\vert^2= -l$. So, if $\varphi$-sectional curvature $c\notin (-5, \frac{1}{2})$ then $l  >0$ which is impossible.
\end{proof}

\section{Ricci-biharmonic pseudo Hopf hypersurfaces}
Let $x: M^{2m}\rightarrow \overline{M}^{2m+1}(c)$ be an isometric immersion from a real hypersurface $M^{2m}$ in a Sasakian space form $\overline{M}^{2m+1}(c)$ with the induced metric of the Ricci tensor and the associated connection $\nabla^{\star}$. Let $\overline{\nabla}$ and $\nabla$ denote the Levi-Civita connection on $\overline{M}^{2m+1}(c)$ and $M^{2m}$, respectively. Let $N$ be a local unit normal vector field on $M^{2m}$ in such away that $\xi$ and $V=-\varphi N$ are tangent to $M^{2m}$. Then we have $T(M^{2m})=D \oplus D^{\perp}$, where $D$ is a maximal $\varphi$-invariant distribution and $D^{\perp}= \textsf{Span}\{\xi, V \}$. Suppose that the Weingarten operator $A$ satisfies $A D^\perp \subseteq D^\perp$ and $AD \subseteq D$. A hypersurface $M^{2m}$ is called a pseudo-Hopf hypersurface provided that the Weingarten operator $A$ is invariant on $ \textsf{Span}\{V, \xi\}$ (see \cite{Abedi}).  Also, it was supposed that $W_1 , W_2 \in \textsf{Span} \{\xi, V \}$ are the eigenvectors of the Weingarten operator $A$ as $AW_1=\gamma_1 W_1$ and $AW_2= \gamma_2 W_2$ such that
\begin{eqnarray}\label{6.00}
& W_1=\xi \cos\theta + V \sin\theta,  \ \
& W_2=-\xi \sin\theta+ V \cos \theta,
\end{eqnarray}
 for some $ 0 < \theta <\frac{\pi}{2}$, where $\gamma_1=-\tan \theta$ and $\gamma_2= \cot \theta$. Let, $AV= \alpha \xi+ \beta V$, then we have $\alpha=-1$ and $\beta= \frac{\cos 2\theta}{\cos \theta \sin \theta}$.

\begin{lem}
Let $(M^{2m},g)$ be a hypersurface in the Sasakian space form $(\overline{M}^{2m+1}(c), \overline{g})$ where the ambient manifold is equipped with the induced metric of the tensor Ricci. Then the Codazzi equation holds 
\begin{eqnarray}\label{7.23}
&&g((\nabla_X A)Y-(\nabla_Y A)X, Z)\nonumber\\
&=& -k_2\frac{c-1}{4}\overline{g}\big(\overline{g}(\varphi X, Z)Y + 2\overline{g}(\varphi X, Y)Z- \overline{g}(\varphi Y, Z)X, V\big),
\end{eqnarray}
where  $X, Y$ and $Z$ tangent on $M^{2m}$.
\end{lem}
\begin{proof}

 Let $N$ be a local unit normal vector field on $M^{2m}$ in the Sasakian space form $\overline{M}^{2m+1}(c)$ and $\xi\in \Gamma(TM^{2m})$, then from the equations $(\ref{7.15})$ and $(\ref{7.1})$ we have
\begin{eqnarray*}
\mathsf{Ric}(\overline{R}(X,Y)Z, N)&=&k_1\eta(\overline{R}(X,Y)Z)\eta(N)+k_2\overline{g}(\overline{R}(X,Y)Z, N)\nonumber\\
&=& -k_2\frac{c-1}{4}\overline{g}\big(\overline{g}(\varphi X, Z)Y + 2\overline{g}(\varphi X, Y)Z- \overline{g}(\varphi Y, Z)X, V\big) ,\nonumber\\
&=&g((\nabla_X A)Y-(\nabla_Y A)X, Z),
\end{eqnarray*}
where $\overline{R}$ is the curvature tensor of the ambient manifold.
\end{proof}
\begin{lem}\label{6.32}
Let $M^{2m}$ be a pseudo Hopf hypersurface in the Sasakian space form $\overline{M}^{2m+1}(c)$. If the Weingarten operator $A$ for some $X\in D$ satisfies $AX=\lambda X$, then $\varphi X$ is an eigenvector of the Weingarten operator corresponding to the eigenvalue in such a way that
\begin{eqnarray}
A\varphi X=\frac{\beta \lambda +2-n}{2\lambda-\beta}\varphi X,
\end{eqnarray}
\end{lem}
where $n=-k_2\frac{c-1}{2}$.
\begin{proof}
Let $X, Y\in D$ be the eigenvectors of the Weingarten operator $A$. We consider $AV= -\xi+ \beta V$ and take the covariant derivative of both sides of this, then 
\begin{eqnarray*}
(\nabla_X A)V+ A\nabla_X V= X(\beta)V+ \beta \nabla_X V +\varphi X,
\end{eqnarray*}
where $\nabla$ denotes the Levi-Civita connection on $M^{2m}$ and $\nabla_X V=\varphi A X$. Hence, we get
\begin{eqnarray*}\label{6.18}
g((\nabla_X A)V, Y)+ g(A\varphi AX, Y)= \beta g(\varphi AX, Y)+g(\varphi X, Y),
\end{eqnarray*}
similarly
\begin{eqnarray*}\label{6.19}
g((\nabla_Y A)V, X)+ g(A\varphi AY, X)= \beta g(\varphi AY, X)+g(\varphi Y, X).
\end{eqnarray*}
Since $\varphi$ is skew-symmetric, therefore 
\begin{eqnarray*}
g((\nabla_X A)Y- (\nabla_Y A)X, V)+2g(A\varphi AX, Y)=\nonumber\\
\beta g(\varphi AX, Y) +\beta g(A\varphi X, Y)+2g(\varphi X, Y),
\end{eqnarray*}
then, from the Codazzi equation in  $(\ref{7.23})$ we obtain
\begin{eqnarray*}
-k_2\frac{c-1}{2}g(\varphi X, Y)+2g(A\varphi AX, Y)=\nonumber\\
\beta g(\varphi AX, Y)+\beta g(A\varphi X, Y)+2g(\varphi X, Y),\nonumber
\end{eqnarray*}
consequently, we have
\begin{eqnarray*}
(2\lambda-\beta)g(A\varphi X, Y)=(\beta \lambda +2 + k_2\frac{c-1}{2})g(\varphi X,Y),
\end{eqnarray*}
as it was claimed.
\end{proof}
In the rest of this section the Ricci-biharmonic pseudo Hopf hypersurfaces discuss in more details.
\begin{thm}
There exists no Ricci-biharmonic pseudo Hopf hypersurface, where $\textsf{grad}|H|$ is a principal direction.
\end{thm}
\begin{proof}
 Let $\psi: M^{2m}\rightarrow \overline{M}^{2m+1}(c)$ be an isometric immersion of a Ricci-biharmonic pseudo Hopf hypersurface in the Sasakian space form $\overline{M}^{2m+1}(c)$. Suppose that \textsf{grad}$|H|$ is in the direction of vectors in $D$. Then, by applying Theorem $\ref{7.7}$ directly, we get
\begin{eqnarray}
\left\{
  \begin{array}{ll}
     \hbox{$\Delta^{\perp}H = (|B|^2 +L)H$,} \\
    \hbox{A $\textsf{grad}|H| + m |H|\textsf{grad} |H| =0$,\nonumber}
  \end{array}
\right.
\end{eqnarray}
in other words in this case grad$|H|$ is an eigenvector of the Weingarten operator $A$ corresponding to the eigenvalue $-m|H|$. 
The Lemma $\ref{6.32}$ yields we can choose a suitable orthogonal frame field $\big{\{}e_1, ...,e_{m-1}, e_m=\varphi e_1, ...,e_{2m-2}= \varphi e_{m-1}, e_{2m-1}=W_1, e_{2m}=W_2 \big{\}}$ on $M^{2m}$ in which
\begin{eqnarray}\label{6.21}
&&Ae_i= \lambda_i e_i, \ \ i=1, . . ., m-1\nonumber\\
&&A\varphi e_i= \overline{\lambda_i}\varphi e_i,  \ \ i= 1, ..., m-1\\
&&AW_1=- \gamma_1 W_1 ,\ \ \  AW_2=\gamma_2 W_2\nonumber
\end{eqnarray}
where, $\lambda_i $ and $\overline{\lambda_i}= \frac{\beta \lambda_i +2- n}{2\lambda_i-\beta}$ are the eigenvalues corresponding to the eigenvectors $e_i$ and $\varphi e_i$, respectively. We recall that $\gamma_1=-\tan \theta $ and $\gamma_2=\cot \theta $, consequently we get $\gamma_1 \gamma_2=-1$.
Let $e_1=\frac{ \textsf{grad}|H|}{|\textsf{grad}|H||}$ such that \textsf{grad}$|H|$ is given by
$ \textsf{grad}|H|= \sum_{i=1} ^{2m} e_i(|H|)e_i$.
 Then
 \begin{eqnarray}\label{6.22}
 {e_1(|H|)\neq 0, \ \ \ \ \ \ \ \ \ e_i(|H|)=0, \ \ \ \ \ \ \ \ \ i=2, ..., 2m. }
\end{eqnarray}
 Also, it is  written
\begin{eqnarray}\label{6.23}
& \nabla_{e_i} e_j =\sum_{k=1} ^{2m} \omega_{ij} ^ke_k,
\end{eqnarray}
where $\nabla$ denotes the Levi-Civita connection on $M^{2m}$ and $\omega_{ij} ^k$ is known the connection forms. Computing the compatibility condition $\nabla_{e_k} \langle e_i, e_j\rangle =0,$ yields
\begin{eqnarray*}
\omega_{ki}^i &=& 0,\ \  i = j  \label{6.25}\\
\omega_{ki}^j+\omega_{kj}^i &=0& , \ \ i\neq j, \ \  i,j,k = 1,...,2m.\label{6.26}
\end{eqnarray*}
Moreovere, take $(\ref{6.21})$ and $(\ref{6.23})$ the Codazzi equation implies
 \begin{eqnarray}
 e_k(\lambda_i)e_i +(\lambda_i -\lambda_j)\omega_{ki}^je_j=e_i(\lambda_k)e_k +(\lambda_k-\lambda_j)\omega_{ik}^je_j,\nonumber
 \end{eqnarray}
which follows
 \begin{eqnarray}
e_i(\lambda_j)&=&(\lambda_i- \lambda_j)\omega_{ji}^j\label{6.27} \\
(\lambda_i- \lambda_j)\omega_{ki}^j&=&(\lambda_k -\lambda_j)\omega_{ik}^j,\label{6.28}
 \end{eqnarray}
 for distinct $i,j,k=1,...,2m$. From $\lambda_1=-m|H|$ and $(\ref{6.22})$ we obtain
\begin{eqnarray}
e_1(\lambda_1) \neq 0, \ \ \ \ \ \ \ \  e_i(\lambda_1)=0,\ \ \ \ \ i=2,...,2m,\label{6.29}
\end{eqnarray}
and
\begin{eqnarray}
 0=[e_i, e_j]\lambda_1 = (\nabla_{e_i} e_j - \nabla_{e_j} e_i)\lambda_1,\ \ \ 2 \leq i,j \leq 2m , i \neq j.
\end{eqnarray}
Thus
\begin{eqnarray}
&\omega_{ij}^1 =\omega_{ji}^1, \label{6.30}	
\end{eqnarray}
 for distinct $i,j=2,..., 2m$.
It is claimed that, $\lambda_j \neq \lambda_1$ for $j=2,..., 2m$. Since, if $\lambda_j =\lambda_1$ for $j\neq 1$, utilize $(\ref{6.27})$ and put $i=1$
\begin{eqnarray*}
0= (\lambda_1- \lambda_j)\omega_{j1} ^j=e_1(\lambda_j)=e_1(\lambda_1),
\end{eqnarray*}
which contradicts $(\ref{6.29})$. For $j=1$ and $ k, i\neq 1$ the equation $(\ref{6.28})$ implies
\begin{eqnarray*}
 (\lambda_i -\lambda_1)\omega_{ki}^1=(\lambda_k - \lambda_1)\omega_{ik}^1,
\end{eqnarray*}
then by the above and $(\ref{6.30})$ it follows
\begin{eqnarray}
& \omega_{ij}^1 =0, \ \ \ \ \ \ \ \ \ \ i\neq j, \ \ \ \ \ \ \ i,j= 2,... ,2m ,\label{6.31}
\end{eqnarray}
take $(\ref{6.26})$, we get $\omega_{i1}^j=0$, for  $i\neq j$ , $i,j= 1,...,2m$.
Now to continue our approach we compute
\begin{eqnarray}\label{6.41}
\nabla_{e_m} W_2 &=& \nabla_{e_m} (-\xi \sin\theta +V \cos \theta)\nonumber \\
&=& -e_m(\sin \theta)\xi - \sin (\theta) \nabla_{e_m} \xi + e_m(\cos \theta)V + \cos(\theta) \nabla_{e_m} V\nonumber \\
&=&-e_m(\sin \theta)(W_1 \sin \theta + W_2 \cos \theta)- \sin \theta (-\varphi e_m)\nonumber \\
 &+& e_m(\cos \theta)(W_1 \cos \theta - W_2 \sin \theta) + \cos \theta(\varphi Ae_m)\nonumber \\
&=&(- e_m (\sin \theta)\sin \theta +e_m(\cos \theta)\cos\theta)W_1 -(\sin \theta + \overline{ \lambda_1}\cos \theta)e_1\nonumber \\
&-&(e_m (\sin \theta ) \cos \theta +e_m (\cos \theta)\sin \theta)W_2.
\end{eqnarray}
Moreover, from $(\ref{6.31})$ the connection form $\omega_{m 2m }^1=0 $. Then $(\ref{6.41})$ follows
\begin{eqnarray*}
& 0= \omega_{m 2m }^1= \sin \theta + \overline{\lambda_1}\cos \theta,
\end{eqnarray*}
which yields
\begin{eqnarray}\label{2.12}
\overline{\lambda_1} = -\tan \theta,
\end{eqnarray}
where $\overline{\lambda}_1$ is an eigenvalue of the Weingarten operator corresponding to the eigenvector $e_m=\varphi e_1$. Similarly, we compute $\nabla_{e_m} W_1$ and apply $(\ref{6.31})$ then
\begin{eqnarray*}
0= \omega_{m 2m-1}^1= \cos \theta -\overline{\lambda_1}\sin \theta, \ \
\end{eqnarray*}
 which yields
\begin{eqnarray}\label{2.13}
\overline{\lambda_1}= \cot \theta.
\end{eqnarray}
Finally, $(\ref{2.12})$ and $(\ref{2.13})$ show a contradiction. Hence, Ricci-biharmonic pseudo Hopf hypersurfaces do not exist in the Sasakian space forms where \textsf{grad}$|H|$ is an eigenvector of the Weingarten operator.
\end{proof}
In the remainder of the studying Ricci-biharmonic pseudo Hopf hypersurfaces are considered, where \textsf{grad}$|H|\in D^{\perp}$.
\begin{prop}
Let $M^{2m}$ be a Ricci-biharmonic pseudo Hopf hypersurface where \textsf{grad}$|H|=\alpha \xi+ \beta V$. Then 
\begin{itemize}
  \item $M^{2m}$ is a hypersurface with the constant mean curvature, if $\alpha=0$.
  \item  does not exist any Ricci-biharmonic pseudo Hopf hypersurface, where $\alpha \neq 0$ and the dimensional of hypersurface is big enough for the $\varphi$-sectional curvature $c>1$.
  \item Ricci-biharmonic pseudo Hopf hypersurface has the mean curvature that holds
  \begin{eqnarray*}
|H|=\frac{1}{2} (-\frac{-1 + m (\gamma_1 + \gamma_2)}{
    m^2} \pm
   \sqrt{\frac{-\frac{
     4 (-1 + c) m^2 (1 + m)}{(-1 + c + 3 m + c m)} + (-1 +
      m (\gamma_1+\gamma_2))^2}{m^4}}),
\end{eqnarray*}
for $\alpha \neq 0$ and the eigenvalues $(\gamma_1 + \gamma_2)\neq 0$ 
where $\varphi$-sectional curvature $c < \frac{(1 - 3 m)}{(1 + m)}$ or $1<c$.
\end{itemize}
\end{prop}
\begin{proof}
Take the assumption $\textsf{grad}|H|=\alpha \xi+ \beta V$ then apply the Theorem $\ref{7.7}$ follows 
\begin{eqnarray}\label{7.24}
\left\{
  \begin{array}{ll}
     \hbox{$\Delta^{\perp}H= (|B|^2+ l)H$ ,} \\
     \hbox{A $(\alpha \xi+ \beta V)-\frac{k_1}{k_2}\alpha V + m |H|(\alpha \xi + \beta V)= 0$}.
  \end{array}
\right.
\end{eqnarray}
 By the above and $(\ref{6.00})$ we have
\begin{eqnarray*}
AW_1=\gamma_1W_1
=\cos \theta A\xi+\sin \theta AV,
\end{eqnarray*}
so 
\begin{eqnarray}\label{7.25}
AV= -\xi+(\gamma_1+ \gamma_2)V,
\end{eqnarray}
then, the second term of the equations $(\ref{7.24})$ yields 
\begin{eqnarray}
\left\{
  \begin{array}{ll}
  \hbox{$-\alpha(|H|+\frac{k_1}{k_2})+\beta((\gamma_1 + \gamma_2)+m|H|)=0$} \\
 \hbox{$-\beta +m|H|\alpha=0$}\label{7.26}
  \end{array}
\right.
\end{eqnarray}
hence
\begin{eqnarray}\label{7.27}
\alpha m^2\big(|H|^2+\frac{(\gamma_1 + \gamma_2)m-1}{m^2}|H|-\frac{(m+1)(1-c)}{(m+1)c+3m-1}\big)=0.
\end{eqnarray}
Suppose that $\alpha=0$, then the equation $(\ref{7.26})$ yields $\beta=0$. Then, the assumption shows $\textsf{grad}|H|=0$. It follows $|H|$ is constant. Nevertheless, where $\alpha \neq 0$, a solution of the equation ($\ref{7.27}$) presents that $(\gamma_1+\gamma_2)< \frac{1}{m}$. Here, $(\gamma_1+\gamma_2)\rightarrow 0$, where the dimension of hypersurfaces is big enough. It implies $\theta =\frac{\pi}{4}$. So, under these circumstances the equation ($\ref{7.27}$) yields that there is no Ricci- biharmonic pseudo Hopf hypersurface, where $\varphi$- sectional curvature $c> 1$ because of $|H|^2 \geq 0$. 
If $\alpha \neq 0$ and the dimension of hypersurface is appropriate, then $(\ref{7.27})$ shows that the boundary for $\varphi$-sectional curvature and the mean curvature hold as the claimed.
\end{proof}

\begin{cor}
Let $M^{2m}$ be a Ricci-biharmonic pseudo Hopf hypersurface where $\textsf{grad}|H|=\alpha \xi$. Then, $M^{2m}$ is a minimal hypersurface.
\end{cor}
\begin{proof}
In this case applying the Theorem $\ref{7.7}$ follows
\begin{eqnarray}
\left\{
  \begin{array}{ll}
    \hbox{$\Delta^{\perp}H= (|B|^2+ l)H $\nonumber;} \\
    \hbox{A$ (\alpha \xi)-\frac{k_1}{k_2}\alpha V + m |H|\alpha \xi= 0$.}
  \end{array}
\right.
\end{eqnarray}
since $\textsf{grad}|H|=\alpha \xi$. Then from (\ref{7.18}) and (\ref{7.25}) the second terms of the above implies
\begin{eqnarray*}
\left\{
  \begin{array}{ll}\label{7.28}
    \hbox{$-\alpha(|H|+\frac{(m+1)(1-c)}{(m+1)c+3m-1})=0$,} \\
    \hbox{$ m|H|\alpha=0.$}
  \end{array}
\right.
\end{eqnarray*}
So, if $\alpha \neq 0$ it shows $|H|=0$.
\end{proof}
\begin{cor}
A Ricci-biharmonic pseudo Hopf hypersurface has the constant mean curvature where $\textsf{grad}|H|=\beta V$. 
\end{cor}
\begin{proof}
Applying the Theorem $\ref{7.7}$ implies
\begin{eqnarray}
\left\{
  \begin{array}{ll}\label{7.29}
    \hbox{$\beta(\gamma_1+\gamma_2+ m|H|)=0$,} \\
    \hbox{$ \beta=0,$}
  \end{array}
\right.
\end{eqnarray}
where $\textsf{grad}|H|=\beta V$. It presents that $\beta=0$ and $|H|$ is constant.
\end{proof}

\bibliographystyle{cas-model2-names}



\end{document}